\documentclass[11pt]{article}
\usepackage{amsmath,amsthm,amssymb,amsfonts}
\usepackage{enumerate,amscd,amsxtra}
\usepackage{mathrsfs}
\usepackage{color}
\usepackage[cmtip,all]{xy}

 \normalsize

\newtheoremstyle{theoremstyle}
  {10pt}      
  {5pt}       
  {\itshape}  
  {}          
  {\bfseries} 
  {:}         
  {.5em}      
  {}          

\newtheoremstyle{examplestyle}
  {10pt}      
  {5pt}       
  {}          
  {}          
  {\bfseries} 
  {:}         
  {.5em}      
  {}          

\theoremstyle{theoremstyle}
\newtheorem{theorem}{Theorem}[section]
\newtheorem*{theorem*}{Theorem}
\newtheorem{lemma}[theorem]{Lemma}

\newtheorem*{proposition*}{Proposition}

\newtheorem*{corollary*}{Corollary}

\newtheorem{definition*}{Definition}
\newtheorem{remark}[theorem]{Remark}
\newtheorem{remark*}{Remark}

\newcommand{\bF}{{\mathbf F}}
\newcommand{\bG}{{\mathbb G}}

\newcommand{\bL}{{\mathbb L}}

\newcommand{\caC}{{\mathcal C}}

\newcommand{\caK}{{\mathcal K}}
\newcommand{\caKK}{{\mathscr K}}

\renewcommand{\P}{{\mathbb P}}

\newcommand{\Ab}{\mathsf{Ab}}
\newcommand{\op}{\mathrm{op}}

\newcommand{\Ho}{\mathsf{Ho}}

\newcommand{\BP}{\mathsf{BP}}

\newcommand{\Hom}{\mathsf{Hom}}
\newcommand{\Map}{\mathsf{Map}}

\newcommand{\EE}{\mathsf{E}}

\newcommand{\Ch}{\mathsf{Ch}}

\newcommand{\ff}{\mathsf{f}}

\newcommand{\MGL}{\mathsf{MGL}}
\newcommand{\MBP}{\mathsf{MBP}}

\newcommand{\RRR}{\mathbf{R}}
\newcommand{\LLL}{\mathbf{L}}
\newcommand{\HGamma}{\mathsf{H}\Gamma}
\newcommand{\HH}{\mathsf{H}}

\renewcommand{\HH}{\mathbf{H}}
\newcommand{\SH}{\mathbf{SH}}

\newcommand{\BGL}{\mathsf{B}\mathbf{GL}}
\newcommand{\KGL}{\mathsf{KGL}}
\newcommand{\KO}{\mathsf{KO}}

\newcommand{\kgl}{\mathsf{kgl}}
\newcommand{\LLLL}{\mathsf{L}}
\newcommand{\ML}{\mathsf{ML}}
\newcommand{\ml}{\mathsf{ml}}

\newcommand{\BU}{\mathsf{BU}}
\newcommand{\KU}{\mathsf{KU}}

\newcommand{\MU}{\mathsf{MU}}

\newcommand{\Z}{\mathbf{Z}}
\newcommand{\Q}{\mathbf{Q}}

\newcommand{\colim}{\mathrm{colim}}

\newcommand{\Mod}{{\mathsf {mod}}}

\newcommand{\hocolim}{\mathrm{hocolim}}

\title{{\bf Existence and uniqueness of $E_{\infty}$ structures on motivic $K$-theory spectra}}
\author{Niko Naumann, Markus Spitzweck, Paul Arne {\O}stv{\ae}r}
\date{\today}
\begin{document}
\maketitle
\begin{abstract}
We show that algebraic $K$-theory $\KGL$,
the motivic Adams summand $\ML$ and their connective covers acquire unique $E_{\infty}$ structures 
refining naive multiplicative structures in the motivic stable homotopy category.
The proofs combine $\Gamma$-homology computations and work due to Robinson giving rise to motivic 
obstruction theory.  
As an application we employ a motivic to simplicial delooping argument to show a uniqueness result for 
$E_\infty$ structures on the $K$-theory Nisnevich presheaf of spectra.
\end{abstract}

\section{Introduction}
Motivic homotopy theory intertwines classical algebraic geometry and modern algebraic topology.
In this paper we study obstruction theory for $E_\infty$ structures in the motivic setup.
An $E_\infty$ structure on a motivic spectrum refers to a coherent homotopy commutative multiplication.
Many examples of motivic ring spectra begin life as a commutative monoid in the motivic stable homotopy category.
This is always the case in applications of motivic Landweber exactness \cite{motiviclandweber}. 
In this paper we are interested in the following questions: 
Can the multiplicative structure of a given commutative monoid in the motivic stable homotopy category be refined to an $E_\infty$ ring spectrum ? 
And if such a refinement exists, is it unique ? 
More ambitiously, 
one may try to determine the homotopy type of a suitable classifying space of $E_\infty$ structures. 
The questions of existence and uniqueness of $E_\infty$ structures and their many ramifications have been studied extensively in topology for the
last forty years \cite{may}.
The first motivic examples of unique $E_\infty$ structures worked out in this paper are of $K$-theoretic interest.
We also prove a less streamlined uniqueness result for $E_\infty$ structures on the $K$-theory Nisnevich presheaf of spectra.
\vspace{0.1in} 

The complex cobordism spectrum $\MU$ and its motivic analogue $\MGL$ have natural $E_\infty$ structures.
In the topological setup, 
Baker and Richter \cite{bakerrichter} have shown that the complex $K$-theory spectrum $\KU$,
the Adams summand $\LLLL$ and the real $K$-theory spectrum $\KO$ admit unique $E_\infty$ structures.
The results in \cite{bakerrichter} are approached via the obstruction theory developed by Robinson in \cite{Robinson}, 
where it is shown that existence and uniqueness of $E_\infty$ structures are guaranteed provided certain $\Gamma$-cohomology groups vanish. 
\vspace{0.1in} 

In our approach we introduce and apply motivic obstruction theory.
Theorem \ref{gammacomputation} shows the relevant motivic $\Gamma$-cohomology groups vanish for algebraic $K$-theory $\KGL$.
The same result holds with the motivic Adams summand $\ML$ introduced in \S \ref{section:ThemotivicAdamssummandsMLandml}. 
The main ingredients in the proofs are our new computations of the $\Gamma$-homology complexes of the topological spectra $\KU$ and $\LLLL$, 
see Theorem \ref{k-theory-coops-gamma-cotangentcomplex} and Lemma \ref{lemma:L_0L},
and the Landweber base change formula for the motivic cooperations of $\KGL$ and $\ML$.
Throughout we work over a fixed separated noetherian base scheme of finite Krull dimension.
Our main result for $\KGL$ can be formulated in the following way.

\begin{theorem}\label{uniqueKGL}
Algebraic $K$-theory $\KGL$ acquires a unique $E_{\infty}$ structure refining its multiplication in the motivic stable homotopy category.
\end{theorem}

While the existence of an $E_\infty$ structure on $\KGL$ is known thanks to the Bott inverted model for algebraic $K$-theory, 
see \cite{RSO}, \cite{SO:Bottinverted}, \cite{gepnersnaith}, 
the analogous existence result for $\ML$ is new. 
The uniqueness part of the theorem rules out the existence of exotic $E_\infty$ structures on $\KGL$.
We note that related motivic $E_\infty$ structures are used in the recent construction of the Atiyah-Hirzebruch spectral sequence for 
motivic twisted $K$-theory \cite{SO:twistedKtheory}.
\vspace{0.1in} 

In Section \ref{juytd} we show that the connective cover $\kgl$ of algebraic $K$-theory has a unique $E_{\infty}$ structure.
The exact same result is shown for the connective cover of the Adams summand in Section \ref{section:ThemotivicAdamssummandsMLandml}.
See \cite{bakerrichterconnective} for the analogous topological results. 
\vspace{0.1in} 

We conclude the introduction with a short overview of the paper.
Section \ref{motivicobs} gives the motivic version of a main result of Robinson's obstruction theory,
see Theorem \ref{motivicgamma}, 
which is applied in our proof of Theorem \ref{uniqueKGL}.
In Section \ref{computeKGL} we show the basic input in motivic obstruction theory is computable for $\KGL$;
our main result is Theorem \ref{gammacomputation}.
Sections \ref{connective} and \ref{adams} treat the examples of connective algebraic $K$-theory and the motivic Adams summand,
respectively.
In Section \ref{multstructures} we use Theorem \ref{motivicgamma} to study multiplicative structures on the $K$-theory Nisnevich presheaf of spectra.
The argument involves an intricate delooping argument allowing to transfer the motivic $E_{\infty}$ structure result to the classical setup. 
We note that existence of an $A_\infty$ structure on a motivic symmetric ring spectrum representing algebraic $K$-theory over a regular base scheme 
was shown in \cite{youngsoo-kim}.

\section{Motivic obstruction theory}
\label{motivicobs}
In this section we prove a basic result in motivic obstruction theory, 
Theorem \ref{motivicgamma}, 
which we use to prove our main results.
Fix a separated Noetherian base scheme $S$ of finite Krull dimension with motivic stable homotopy category $\SH(S)$.
Motivic functors \cite{DRO}, $S$-modules \cite{hu}, 
and motivic symmetric spectra \cite{jardinemotivicspectra} furnish monoidal model structures whose associated homotopy categories 
are equivalent to $\SH(S)$.

Let $\EE$ be a homotopy commutative motivic ring spectrum, 
i.e., a commutative and associative unitary monoid in $\SH(S)$. 
Set $R\equiv\EE_{**}$ and $\Lambda\equiv\EE_{**}\EE$, 
the coefficients and cooperations of $\EE$,
respectively. 
Here we write as usual $\EE_{p,q}:=\SH(S)(S^{p,q},\EE)$ for the motivic spheres $S^{p,q}:=(S^1)^{p-q}\wedge (\bG_m)^{\wedge q}$ ($p,q\in\mathbb{Z}$).

Then $\EE$ satisfies UCT, 
the universal coefficient theorem, 
if for all $n\ge 1$ the Kronecker product yields an isomorphism
\[ 
\EE^{**}(\EE^{\wedge\,n})
\stackrel{\cong}{\longrightarrow} 
\Hom_R(\Lambda^{\otimes_{R} n},R).
\]
It is well known that algebraic $K$-theory $\KGL$ satisfies UCT \cite[Theorem 9.3 (i)]{motiviclandweber}. 

For $\EE$ as above, 
one associates to $R$ and $\Lambda$ the $\Gamma$-homology complex and the trigraded motivic $\Gamma$-cohomology 
groups $\HGamma^{***}(\Lambda|R;R)$ of $\Lambda$ over $R$, 
cf.~Section \ref{computeKGL} for details. 
When read from left to right, 
the trigrading is given by the Gamma-cohomology degree, the simplicial degree and the weight.
Our motivic version of Robinson's \cite[Theorem 5.6]{Robinson} takes the following form.

\begin{theorem}
\label{motivicgamma}
Suppose $\EE$ is a homotopy commutative motivic ring spectrum satisfying UCT, 
$\HGamma^{n,2-n,0}(\Lambda|R;R)=0$ for $n \ge 4$, 
and $\HGamma^{n,1-n,0}(\Lambda|R;R)=0$ for $n\ge 3$.
Then $\EE$ acquires an $E_\infty$ structure which is unique up to homotopy.
\end{theorem}

The proof of Theorem \ref{motivicgamma} is a straightforward adaption of Robinson's work in \cite{Robinson} and 
\cite{robinsonrevisited} to the motivic setup. 
Following \cite{Robinson} and \cite{robinsonrevisited} we fix a specific cofibrant $E_\infty$ operad $E\Sigma\times{\mathcal T}$ in simplicial sets.
The notion of stages yields the derived mapping space ${\mathcal E}_\infty({\mathcal C},\mu)$ from $E\Sigma\times {\mathcal T}$ to an (almost)
arbitrary simplicial operad $\caC$ as the homotopy limit of an explicit tower, 
cf.~\cite[Section 4.4]{robinsonrevisited}. 
The tower sets up a Bousfield-Kan spectral sequence, taking the form
\[ 
E_1^{pq}
=
\pi_{q-p}St^p_{p-1}
\Longrightarrow 
\pi_{q-p}{\mathcal E}_\infty({\mathcal C},\mu).
\]
The filtration quotients $St^p_{p-1}$ are determined by the stages.
An analysis of the geometry of $E\Sigma\times {\cal T}$ shows the $E_2$-page coincides with the stable cohomotopy of the $\Gamma$-module 
afforded by the homotopy of $\cal C$.
Our main example of interest is the endomorphism operad $\caC$ of $\EE$, 
formed in any of the simplicial model categories underlying $\SH(S)$,
with spaces $\caC_n=\Map(\EE^{\wedge n},\EE)$. 
By the UCT we obtain
\[ 
\pi_*\caC_\bullet
=
\EE^{*,0}(\EE^{\wedge \bullet})\subseteq \EE^{**}(\EE^{\wedge \bullet})
\simeq
\Hom_{\EE_{**}}(\EE_{**}\EE^{\otimes_{\EE_{**}}\bullet},\EE_{**}).
\]
Following verbatim the arguments leading up to \cite[Corollary 3.7 and Theorem 4.2]{Robinson}, 
while carrying along the additional grading given by the motivic weight, 
shows this is an identification of $\Gamma$-modules. 
This observation concludes the proof of Theorem \ref{motivicgamma}.

We remark that, 
while Theorem \ref{motivicgamma} only concerns the weight zero part of motivic $\Gamma$-cohomology,
this cannot a priori be extracted exclusively from the weight zero parts of the algebras $\EE_{**}$ and $\EE_{**}\EE$.

Our proof of Theorem \ref{uniqueKGL} is complete by combining Theorem \ref{motivicgamma} for $\EE=\KGL$ and Theorem \ref{gammacomputation}.

\section{Algebraic $K$-theory $\KGL$}\label{computeKGL}
In this section we shall present the $\Gamma$-cohomology computation showing there is a unique $E_{\infty}$ structure on the algebraic $K$-theory spectrum $\KGL$.
Throughout we work over some separated Noetherian base scheme of finite Krull dimension, 
which will be omitted from the notation.

There are two main ingredients which make this computation possible: 
First, the $\Gamma$-homology computation of $\KU_0\KU$ over $\KU_0=\Z$, 
where $\KU$ is the complex $K$-theory spectrum. 
Second, 
we employ base change for the motivic cooperations of algebraic $K$-theory, 
as shown in our previous work \cite{motiviclandweber}.

\subsection{The $\Gamma$-homology of $\KU_0\KU$ over $\KU_0$}
\label{subsection:TheGamma-homologyofKU_0KUoverKU_0}
For a map $A\longrightarrow B$ between commutative algebras we denote Robinson's $\Gamma$-homology complex by 
$\widetilde{\caK}(B\vert A)$ \cite[Definition 4.1]{Robinson}.
Recall that $\widetilde{\caK}(B\vert A)$ is a homological double complex of $B$-modules concentrated in the first quadrant.
The same construction can be performed for maps between graded and bigraded algebras. 
In all cases we let ${\caK}(B\vert A)$ denote the total complex associated with the double complex $\widetilde{\caK}(B\vert A)$.
\vspace{0.1in} 

The $\Gamma$-cohomology
$$
\HGamma^*(\KU_0\KU|\KU_0, -)
=
\HH^*\RRR\Hom_{\KU_0\KU}({\caK}(\KU_0\KU|\KU_0),-) 
$$ 
has been computed for various coefficients in \cite{bakerrichter}. 
In what follows we require precise information about the complex ${\caK}(\KU_0\KU|\KU_0)$ itself, 
since it is the object which satisfies a motivic base change property, 
cf.~Lemma \ref{coops}.
\begin{lemma}
\label{rational}
Let $X\in \Ch_{\ge 0}(\Ab)$ be a non-negative chain complex of abelian groups. 
The following statements are equivalent.
\begin{enumerate}
\item[i)] The canonical map $X\longrightarrow X\otimes_\Z^{\LLL} \Q=X\otimes_\Z\Q$ is a quasi-isomorphism.
\item[ii)] For every prime $p$, there is a quasi-isomorphism $X\otimes ^{\LLL}_\Z\bF_p\simeq 0$.
\end{enumerate}
\end{lemma}
\begin{proof} 
It is well known that $X$ is formal \cite[pg.~164]{goerssjardine}, 
i.e., 
there is a quasi-isomorphism
\[ 
X\simeq \bigoplus_{n\ge 0}H_n(X)[n]. 
\]
Here, 
for an abelian group $A$ and integer $n$, 
we let $A[n]$ denote the chain complex that consists of $A$ concentrated in degree $n$.
Hence for every prime $p$, 
\[ 
X\otimes^{\LLL}_\Z \bF_p\simeq \bigoplus_{n\ge 0}\left( H_n(X)[n]\otimes^{\LLL}_\Z\bF_p\right).
\]
By resolving $\bF_p=(\Z\stackrel{\cdot p}{\longrightarrow}\Z)$ one finds an isomorphism
\[ 
H_*(A[n]\otimes_{\Z}^{\LLL}\bF_p)
\cong
(A/pA)[n]\oplus A\{p\}[n+1]
\]
for every abelian group $A$ and integer $n$.
Here $A\{p\}$ is shorthand for $\{ x\in A\, |\, px=0\}$. 
In summary, 
ii) holds if and only if the multiplication by $p$ map
\[ 
\cdot 
p
\colon
H_*(X)
\longrightarrow 
H_*(X)
\]
is an isomorphism for every prime $p$. The latter is equivalent to i).
\end{proof} 

We shall use the previous lemma in order to study cotangent complexes introduced by Illusie in \cite{illusie}. 
Let $R$ be a ring and set $R_\Q:=R\otimes_\Z\Q$.
There is a canonical map 
\[ 
\tau_R: \bL_{R/\Z}
\longrightarrow
\bL_{R/\Z}\otimes_\Z^{\LLL}\Q\simeq 
\bL_{R/\Z}\otimes_R^{\LLL} R_\Q
\stackrel{\simeq}{\longrightarrow}
\bL_{R_\Q/\Q} 
\]
between cotangent complexes in $\Ho(\Ch_{\ge 0}(\Z))$.
The first quasi-isomorphism is obvious, 
while the second one is an instance of flat base change for cotangent complexes.
\begin{lemma}\label{test}
The following statements are equivalent.
\begin{enumerate}
\item[i)] $\tau_R$ is a quasi-isomorphism.
\item[ii)] For every prime $p$, 
there is a quasi-isomorphism $\bL_{R/\Z}\otimes_\Z^{\LLL}\bF_p\simeq 0$.
\end{enumerate}
If the abelian group underlying $R$ is torsion free, 
then i) and ii) are equivalent to
\begin{enumerate}
\item[iii)] For every prime $p$, $\bL_{(R/pR)/\bF_p}\simeq 0$.
\end{enumerate}
\end{lemma}
\begin{proof} 
The equivalence of i) and ii) follows by applying Lemma \ref{rational} to $X=\bL_{R/\Z}$. 
If $R$ is torsion free, 
then it is flat as a $\Z$-algebra.
Hence, 
by flat base change, 
there exists a quasi-isomorphism
\[ 
\bL_{R/\Z}\otimes_\Z^{\LLL}\bF_p\simeq \bL_{(R/pR)/\bF_p}.
\]
\end{proof}

The following is our analogue for Robinson's $\Gamma$-homology complex of the Baker-Richter result \cite[Theorem 5.1]{bakerrichter}.
\begin{theorem} 
\label{k-theory-coops-gamma-cotangentcomplex}
\begin{itemize}
\item [i)] Let $R$ be a torsion free ring such that $\bL_{(R/pR) / \bF_p}\simeq 0$ for every prime $p$, 
e.g., assume that $\bF_p\longrightarrow R/pR$ is ind-\'etale for all $p$.
Then there is a quasi-isomorphism
\[ 
{\caK}(R|\Z)\simeq {\caK}(R_\Q|\Q)
\]
in the derived category of $R$-modules.
\item[ii)] There is a quasi-isomorphism
\[ 
{\caK}(\KU_0\KU | \KU_0)\simeq (\KU_0\KU)_\Q[0]
\]
in the derived category of $\KU_0\KU$-modules.
\end{itemize}\end{theorem}
\begin{proof}
\begin{itemize}\item[i)] 
The Atiyah-Hirzebruch spectral sequence noted in \cite[Remark 2.3]{richter} takes the form
\[ 
E^2_{p,q}=H^p(\bL_{R/\Z} \otimes_\Z^{\LLL}  \Gamma^q(\Z[x] |\Z))\Longrightarrow H^{p+q}({\caK}(R|\Z)).
\]
Our assumptions on $R$ and Lemma \ref{test} imply that the $E^2$-page is comprised of $\Q$-vector spaces. 
Hence so is the abutment, 
and there exists a quasi-isomorphism between complexes of $R$-modules
\[  
{\caK}(R|\Z)\stackrel{\simeq}{\longrightarrow}  {\caK}(R|\Z)\otimes_\Z\Q. 
\]
Moreover, 
by Lemma \ref{naivebasechange}, 
there is a quasi-isomorphism
\[{\caK}(R|\Z)\otimes_\Z\Q \simeq {\caK}(R_\Q|\Q).\]
\item[ii)] 
By \cite[Theorem 3.1, Corollary 3.4, (a)]{bakerrichter} and the Hopf algebra isomorphism
$A^{st}\simeq \KU_0\KU$ \cite[Proposition 6.1]{bakerrichter}, 
the ring $\KU_0\KU$ satisfies the assumptions for $R$ of part i)\footnote{Alternatively, this follows easily from Landweber exactness of $\KU$.}.
Now since $\KU_0\cong\Z$, 
\[  
{\caK}(\KU_0\KU | \KU_0)\simeq {\caK}((\KU_0\KU)_\Q | \Q).
\]
We have that $(\KU_0\KU)_\Q\simeq\Q[w^{\pm 1}]$ \cite[Theorem 3.2, (c)]{bakerrichter} is a smooth $\Q$-algebra.
Hence, 
since $\Gamma$-cohomology agrees with Andr\'e-Quillen cohomology over $\Q$,
there are quasi-isomorphisms 
\[ 
{\caK}(\KU_0\KU | \KU_0)
\simeq 
\Omega^1_{\Q[w^{\pm 1}]| \Q}[0]
\simeq
(\KU_0\KU)_\Q[0].
\]
\end{itemize}
\end{proof}

\subsection{The $\Gamma$-homology of $\KGL_{\ast\ast} \KGL$ over $\KGL_{\ast\ast}$}
The strategy in what follows is to combine the computations for $\KU$ in \S\ref{subsection:TheGamma-homologyofKU_0KUoverKU_0} 
with motivic Landweber exactness \cite{motiviclandweber}.
To that end we require the following general base change result,
which was also used in the proof of Theorem \ref{k-theory-coops-gamma-cotangentcomplex}.
\begin{lemma}
\label{naivebasechange}
For a pushout of ordinary, graded or bigraded commutative algebras
\begin{equation*}
\xymatrix{
A\ar[d]\ar[r] & B\ar[d] \\
C\ar[r] & D }
\end{equation*}
there are isomorphisms between complexes of $D$-modules 
$$
{\caK}(D\vert C)
\cong
{\caK}(B\vert A)\otimes_{B}D
\cong
{\caK}(B\vert A)\otimes_{A}C.
$$

If $B$ is flat over $A$, 
then $\widetilde{\caK}(B\vert A)$ is a first quadrant homological double complex of flat $B$-modules;
thus, in the derived category of $D$-modules there are quasi-isomorphisms
$$
{\caK}(D\vert C)
\simeq
{\caK}(B\vert A)\otimes^{\LLL}_{B}D
\simeq
{\caK}(B\vert A)\otimes^{\LLL}_{A}C.
$$
\end{lemma}
\begin{proof} 
Following the notation in \cite[\S4]{Robinson},
let $(B\vert A)^{\otimes}$ denote the tensor algebra of $B$ over $A$.
Then $(B\vert A)^{\otimes}\otimes_{A}B$ has a natural $\Gamma$-module structure over $B$,
cf.~\cite[\S4]{Robinson}.
Here $\Gamma$ denotes the category of finite based sets and basepoint preserving maps.
It follows that $((B\vert A)^{\otimes}\otimes_{A}B)\otimes_{B}D$ is a $\Gamma$-module over $D$.
Moreover,
by base change for tensor algebras, 
there exists an isomorphism of $\Gamma$-modules in $D$-modules
$$
((B\vert A)^{\otimes}\otimes_{A}B)\otimes_{B}D
\cong
(D\vert C)^{\otimes}\otimes_{C}D.
$$
Here we use that the $\Gamma$-module structure on $(B\vert A)^{\otimes}\otimes_{A} M$,
for $M$ a $B$-module, is given as follows: 
For a map $\varphi \colon [m] \to [n]$ between finite pointed sets, 
$$
(B \otimes_A B \otimes_A \cdots \otimes_A B)\otimes_A M 
\longrightarrow 
(B \otimes_A B \otimes_A \cdots \otimes_A B) \otimes_A M
$$
sends $b_1 \otimes \cdots \otimes b_m \otimes m$ to 
$$
(\prod_{i \in \varphi^{-1}(1)} b_i) \otimes \cdots \otimes (\prod_{i \in \varphi^{-1}(n)} b_i) \otimes ((\prod_{i \in \varphi^{-1}(0)} b_i) \cdot m).
$$
By convention, 
if $\varphi^{-1}(j)=\emptyset$ then $\prod_{i \in \varphi^{-1}(j)} b_i=1$.
Robinson's $\Xi$-construction yields an isomorphism between double complexes of $D$-modules 
$$
\widetilde{\caK}(D\vert C)=
\Xi((D\vert C)^{\otimes}\otimes_{C}D)
\cong
\Xi(((B\vert A)^{\otimes}\otimes_{A}B)\otimes_{B}D).
$$
Inspection of the $\Xi$-construction reveals there is an isomorphism 
$$
\Xi(((B\vert A)^{\otimes}\otimes_{A}B)\otimes_{B}D)
\cong
\Xi((B\vert A)^{\otimes}\otimes_{A}B)\otimes_{B}D.
$$
By definition, 
this double complex of $D$-modules is $\widetilde{\caK}(B\vert A)\otimes_{B}D\cong\widetilde{\caK}(B\vert A)\otimes_{A}C$. 
This proves the first assertion by comparing the corresponding total complexes.
The remaining claims follow easily.
\end{proof}

Next we recall the structure of  the motivic cooperations of the algebraic $K$-theory spectrum $\KGL$. 
The algebras we shall consider are bigraded as follows: 
$\KU_0\cong\Z$ in bidegree $(0,0)$ and $\KU_*\cong\Z[\beta^{\pm 1}]$ with the Bott element $\beta$ in bidegree $(2,1)$. 
With these conventions, 
there is a canonical bigraded map 
\[
\KU_*
\longrightarrow 
\KGL_{\ast\ast}.
\] 
\begin{lemma}\label{coops}
There are pushouts of bigraded algebras 
\begin{equation*}
\xymatrix{
\KU_{\ast} \ar[d]\ar[r]^-{\eta_L}  & \KU_{\ast}\KU \ar[d] \\
\KGL_{\ast\ast} \ar[r]^-{\eta_L}  & \KGL_{\ast\ast}\KGL }
\;\;\;\;\;\;
\xymatrix{
\KU_{0} \ar[d]\ar[r]^-{(\eta_L)_0} & \KU_{0}\KU\ar[d] \\
\KU_{\ast}\ar[r]^-{\eta_L} & \KU_{\ast}\KU }
\end{equation*}
and a quasi-isomorphism in the derived category of
$\KGL_{\ast\ast}\KGL$-modules
$$
{\caK}(\KGL_{\ast\ast}\KGL\vert\KGL_{\ast\ast})
\simeq
{\caK}(\KU_{0}\KU\vert\KU_{0})\otimes^{\LLL}_{\KU_{0}\KU}\KGL_{\ast\ast}\KGL.
$$
\end{lemma}
\begin{proof}
Here, $\eta_L$ is a generic notation for the left unit of some flat Hopf-algebroid. 
The first pushout is shown in \cite[Proposition 9.1 (i)]{motiviclandweber} and the second one in \cite{bakerrichter}. 
Applying Lemma \ref{naivebasechange} twice gives the claimed quasi-isomorphism.
\end{proof}

Next we compute the $\Gamma$-cohomology of the motivic cooperations of $\KGL$.
\begin{theorem}
\begin{itemize}\label{gammacomputation}
\item[i)] There is an isomorphism
\[ 
\HGamma^{\ast,\ast,\ast}(\KGL_{\ast\ast}\KGL\vert\KGL_{\ast\ast};\KGL_{\ast\ast})\cong \HH^*\RRR\Hom_\Z(\Q[0],\KGL_{\ast\ast}).
\]
\item[ii)] For all $s\ge 2$,
\[ 
\HGamma^{s,\ast,\ast}(\KGL_{\ast\ast}\KGL\vert\KGL_{\ast\ast};\KGL_{\ast\ast})=0.
\]
\end{itemize}
\end{theorem}
\begin{proof}
\begin{itemize}
\item[i)] By the definition of $\Gamma$-cohomology and the results in this Subsection there are isomorphisms
\begin{equation*}
\begin{array}{rl} 
& \HGamma^{\ast,\ast,\ast}(\KGL_{\ast\ast}\KGL\vert\KGL_{\ast\ast};\KGL_{\ast\ast})\\
= & \HH^*\RRR\Hom_{\KGL_{\ast\ast}\KGL}({\caK}(\KGL_{\ast\ast}\KGL\vert\KGL_{\ast\ast}),\KGL_{\ast\ast}) \\
\cong & \HH^*\RRR\Hom_{\KGL_{\ast\ast}\KGL}({\caK}(\KU_{0}\KU\vert\Z)\otimes^{\LLL}_{\KU_{0}\KU}\KGL_{\ast\ast}\KGL,\KGL_{\ast\ast})\\
\cong & \HH^*\RRR\Hom_{\KU_{0}\KU}({\caK}(\KU_{0}\KU\vert\Z),\KGL_{\ast\ast})\\
\cong & \HH^*\RRR\Hom_{\KU_{0}\KU}((\KU_{0}\KU)_{\Q}[0],\KGL_{\ast\ast})\\
\cong & \HH^*\RRR\Hom_{\Z}(\Q[0],\KGL_{\ast\ast}).\\
\end{array}
\end{equation*}
\item[ii)] This follows from i) since $\Z$ has global dimension $1$.
\end{itemize}
\end{proof}
\begin{remark}
It is an exercise to compute $\RRR\Hom_\Z(\Q,-)$ applied to finitely generated abelian groups. 
This makes our $\Gamma$-cohomology computation explicit in cohomological degrees $0$ and $1$ for base schemes with finitely generated algebraic $K$-groups,
e.g., finite fields and number rings.
The computation $\RRR\Hom_\Z(\Q,\Z)\simeq{\hat{\Z}/\Z}[1]$ shows our results imply \cite[Corollary 5.2]{bakerrichter}.
\end{remark}

\section{Connective algebraic $K$-theory $\kgl$}\label{connective}
\label{juytd}
We define the connective algebraic $K$-theory spectrum $\kgl$ as the effective part $\ff_{0}\KGL$ of $\KGL$.
Recall the connective cover functor $\ff_i$ defined in \cite{voevodskyopen} projects from the motivic stable homotopy category to its $i$th effective part.
Note that $\ff_{0}\KGL$ is a commutative monoid in the motivic stable homotopy category. 
This holds since $\ff_{0}$ is a lax symmetric monoidal functor; 
indeed, 
it is right adjoint to a monoidal functor.
For $i\in\Z$ there exists a natural map $\ff_{i+1}\KGL\rightarrow\ff_{i}\KGL$ in $\SH(S)$ with cofiber the $i$th slice of $\KGL$.
We have in particular $\KGL\cong\hocolim\,\ff_{i}\KGL$,  
cf., \cite[Lemma 4.2]{voevodskyopen}.
Bott periodicity for algebraic $K$-theory implies that $\ff_{i+1}\KGL\cong\Sigma^{2,1}\ff_{i}\KGL$.      
This allows to recast the colimit as $\hocolim\,\Sigma^{2i,i}\kgl$ with multiplication by the Bott element $\beta$ in $\kgl^{-2,-1}\cong\KGL^{-2,-1}$ as the transition map at each stage.
We summarize these observations in a lemma.
\begin{lemma}
\label{lemma:KGLbottkgl}
The algebraic $K$-theory spectrum $\KGL$ is isomorphic in the motivic stable homotopy category to the Bott inverted connective algebraic $K$-theory spectrum $\kgl[\beta^{-1}]$.
\end{lemma}

\begin{theorem}
The connective algebraic $K$-theory spectrum $\kgl$ has a unique $E_{\infty}$ structure refining its multiplication in the motivic stable homotopy category.
\end{theorem}
\begin{proof}
The connective cover $\ff_{0}$ preserves $E_{\infty}$ structures \cite[\S 6]{GRSO}.
Thus the existence of an $E_{\infty}$ structure on $\kgl$ is ensured.
We note that inverting the Bott element can be refined to the level of motivic $E_{\infty}$ ring spectra by the methods employed in \cite{RSO}.
Thus,
by Lemma \ref{lemma:KGLbottkgl}, 
starting out with any two $E_{\infty}$ structures on $\kgl$ produces two $E_{\infty}$ structures on $\KGL$, 
which coincide by the uniqueness result for $E_{\infty}$ structures on $\KGL$ .
Applying $\ff_{0}$ recovers the two given $E_{\infty}$ structures on $\kgl$: 
If $\kgl^{\prime}$ is $E_\infty$ with $\varphi \colon \kgl\simeq\kgl^{\prime}$ as ring spectra, 
then there is a canonical $E_\infty$ map $\kgl^{\prime} \to \kgl^{\prime}[{\beta'}^{-1}]$, 
where $\beta'$ is the image of the Bott element under $\varphi$. 
Since $\kgl^{\prime}$ is an effective motivic spectrum,
this map factors as an $E_\infty$ map $\kgl^{\prime} \to \ff_0(\kgl^{\prime}[\beta'^{-1}])$. 
By construction of $\kgl$ the latter map is an equivalence.
This shows the two given $E_\infty$ structures on $\kgl$ coincide.
\end{proof}

\section{The motivic Adams summands $\ML$ and $\ml$}\label{adams}
\label{section:ThemotivicAdamssummandsMLandml}
Let $\BP$ denote the Brown-Peterson spectrum for a fixed prime number $p$.
Then the coefficient ring $\KU_{(p)\ast}$ of the $p$-localized complex $K$-theory spectrum is a $\BP_{\ast}$-module via the ring map $\BP_{\ast}\rightarrow\MU_{(p)\ast}$ 
which classifies the $p$-typicalization of the formal group law over $\MU_{(p)\ast}$. 
The $\MU_{(p)\ast}$-algebra structure on $\KU_{(p)\ast}$ is induced from the natural orientation $\MU \to \KU$.
With this $\BP_{\ast}$-module structure, 
$\KU_{(p)\ast}$ splits into a direct sum of the suspensions $\Sigma^{2i}\LLLL_{\ast}$ for $0\leq i\leq p-2$, 
where $\LLLL$ is the Adams summand of $\KU_{(p)}$.
Thus motivic Landweber exactness \cite{motiviclandweber} over the motivic Brown-Peterson spectrum $\MBP$ produces a splitting of motivic spectra
$$
\KGL_{(p)}
=
\bigvee_{i=0}^{p-2}
\Sigma^{2i,i}\ML.
$$
We refer to $\ML$ as the motivic Adams summand of algebraic $K$-theory.

Since $\LLLL_{\ast}$ is an $\BP_{\ast}$-algebra and there are no nontrivial phantom maps from any smash power of $\ML$ to $\ML$, 
which follows from \cite[Remark 9.8, (ii)]{motiviclandweber} since $\ML$ is a retract of $\KGL_{(p)}$,
we deduce that the corresponding ring homology theory induces a commutative monoid structure on $\ML$ in the motivic stable homotopy category. 

We define the connective motivic Adams summand $\ml$ to be $\ff_{0}\ML$.
It is also a commutative monoid in the motivic homotopy category.
\begin{theorem}
\label{theorem:MLml}
The motivic Adams summand $\ML$ has a unique $E_{\infty}$ structure refining its multiplication in the motivic stable homotopy category.
The same result holds for the connective motivic Adams summand $\ml$. 
\end{theorem}

The construction of $\ML$ as a motivic Landweber exact spectrum makes the following result evident on account of 
the proof of Lemma \ref{coops}.
\begin{lemma}
\label{coopsII}
There exist pushout squares of bigraded algebras 
\begin{equation*}
\xymatrix{
\LLLL_{\ast} \ar[d]\ar[r]^-{\eta_L}  & \LLLL_{\ast}\LLLL \ar[d] \\
\ML_{\ast\ast} \ar[r]^-{\eta_L}  & \ML_{\ast\ast}\ML }
\;\;\;\;\;\;
\xymatrix{
\LLLL_{0} \ar[d]\ar[r]^-{(\eta_L)_0} & \LLLL_{0}\LLLL\ar[d] \\
\LLLL_{\ast}\ar[r]^-{\eta_L} & \LLLL_{\ast}\LLLL }
\end{equation*}
and a quasi-isomorphism in the derived category of $\ML_{\ast\ast}\ML$-modules
$$
{\caK}(\ML_{\ast\ast}\ML\vert\ML_{\ast\ast})
\simeq
{\caK}(\LLLL_{0}\LLLL\vert\LLLL_{0})\otimes^{\LLL}_{\LLLL_{0}\LLLL}\ML_{\ast\ast}\ML.
$$
\end{lemma}

Next we show the analog of Theorem \ref{k-theory-coops-gamma-cotangentcomplex}, ii) for the motivic Adams summand.
\begin{lemma} 
\label{lemma:L_0L}
In the derived category of $\LLLL_0\LLLL$-modules, there is a quasi-isomorphism
$$
{\caK}(\LLLL_0\LLLL|\LLLL_0)
\simeq 
(\LLLL_0\LLLL)_\Q[0].
$$
\end{lemma}
\begin{proof}
In the notation of \cite[Proposition 6.1]{bakerrichter} there is an isomorphism between Hopf algebras $\LLLL_0\LLLL\cong {}^{\zeta}A^{st}_{(p)}$.
Recall that ${}^{\zeta}A^{st}_{(p)}$ is a free $\Z_{(p)}$-module on a countable basis and ${}^{\zeta}A^{st}_{(p)}/p{}^{\zeta}A^{st}_{(p)}$ is a formally 
\'etale $\bF_{p}$-algebra \cite[Theorem 3.3(c), Corollary 4.2]{bakerrichter}. 
Applying Theorem \ref{k-theory-coops-gamma-cotangentcomplex}, i) to $R=\LLLL_0\LLLL$ and using that $(\LLLL_0\LLLL)_\Q\simeq\Q[v^{\pm 1}]$ 
by Landweber exactness, 
where $v=w^{p-1}$ and $(\KU_0 \KU)_\Q\cong\Q[w^{\pm 1}]$,
we find
\[ {\caK}(\LLLL_0\LLLL | \LLLL_0)
\simeq 
\Omega^1_{\Q[v^{\pm 1}]| \Q}[0]
\simeq 
(\LLLL_0\LLLL)_\Q[0].\]
\end{proof}

Lemmas \ref{coopsII} and \ref{lemma:L_0L} imply there is a quasi-isomorphism
\[ 
\HGamma^{\ast,\ast,\ast}(\ML_{\ast\ast}\ML\vert\ML_{\ast\ast};\ML_{\ast\ast})\simeq \HH^*\RRR\Hom_\Z(\Q[0],\ML_{\ast\ast}).
\]
Thus the part of Theorem \ref{theorem:MLml} dealing with $\ML$ follows, since for all $s\ge 2$ we have
\begin{equation}
\label{gammahomologyvanishingofML}
\HGamma^{s,\ast,\ast}(\ML_{\ast\ast}\ML\vert\ML_{\ast\ast};\ML_{\ast\ast})=0.
\end{equation}
The assertion about $\ml$ follows by the exact same type of argument as for $\kgl$.
The periodicity operator in this case is $v_1 \in \ml^{2(1-p),1-p}=\ML^{2(1-p),1-p}$.

\section{Multiplicative structures on the $K$-theory presheaf}
\label{multstructures}

Fix a Noetherian separated and regular scheme $S$ of finite Krull dimension and its Nisnevich site $Sm/S$ of smooth schemes of finite type.
We consider the $K$-theory presheaf of spectra \cite{jardine} taking values in $E_\infty$ rings
\[ 
K
\colon
(Sm/S)^{\op}
\longrightarrow
E_\infty \text{ rings}.
\]
For a smooth $S$-scheme scheme $X$, 
the $E_\infty$ structure on $K(X)$ arises from the bipermutative category of locally free ${\mathcal O}_X$-modules of finite rank \cite{may2}.
One may ask if this $E_\infty$ structure is the unique one refining the underlying homotopy commutative ring spectrum structure on $K(X)$. 
While we are unaware of any (non-empty) scheme $X$ for which a complete answer to this question is known,
we show the following.

\begin{theorem}
Fix $S$ as above and a presheaf
\[ 
\caKK
\colon
(Sm/S)^{\op} 
\longrightarrow 
E_\infty \text{ rings}.
\]
Suppose $\caKK$ and $K$ are equivalent as homotopy commutative $S^1$-spectra.
Then they are equivalent as $E_\infty$ $S^1$-spectra. 
In particular, 
if $X\in Sm/S$, 
there is an equivalence of $E_\infty$ ring spectra $\caKK(X)\simeq K(X)$.
\end{theorem}

\begin{proof}\footnote{The model structures appearing in this proof exist by standard lifting arguments.}
We use the given equivalence $\caKK\simeq K$ and the known Bott periodicity of $K$ to deduce Bott periodicity for $\caKK$, 
i.e., there is a map $\Sigma^\infty\P^1\to\caKK$ of $S^1$-spectra such that 
the adjoint of its $\caKK$-linear extension $b:\Sigma^\infty\P^1\wedge\caKK\to\caKK$ is an equivalence
$\caKK\stackrel{\simeq}{\to}\Omega_{\P^1}\caKK$. 
We will use this to deloop the $E_\infty$ structure on $\caKK$.

Let $\Mod_{S^0}$ denote the category of $S^1$-spectra and $\Mod_\caKK$ the category of $\caKK$-modules in $S^1$-spectra. 
We denote by $\mathsf{CAlg}(\Mod_\caKK^\Sigma)$ the commutative monoids in the category $\Mod_\caKK^\Sigma$ of symmetric sequences in 
$\caKK$-modules (endowed with the Day convolution product).
We note there is an adjunction 
\[ 
\xymatrix{
\mathsf{Sym}_\caKK:\Mod_\caKK \ar@<.5ex>[r]  & 
\mathsf{CAlg}(\Mod_\caKK^\Sigma). \ar@<.5ex>[l] } 
\]
Here $\mathsf{Sym_\caKK}$ sends $M$ to $(M^{\wedge_\caKK n})_{n\ge 0}$ and its right adjoint picks out the first entry of a symmetric sequence.
The map 
\[ 
\mathsf{Sym}_\caKK(b)
\colon
\mathsf{Sym}_\caKK(\Sigma^\infty\P^1\wedge\caKK)\to\mathsf{Sym}_\caKK(\caKK)
\]
lets us consider $\mathsf{Sym}_\caKK(\caKK)$ as a commutative monoid in $\mathsf{Sym}_\caKK(\Sigma^\infty\P^1\wedge\caKK)-\Mod^\Sigma_\caKK$
(the category of $\mathsf{Sym}_\caKK(\Sigma^\infty\P^1\wedge\caKK)$-modules in $\Mod^\Sigma_\caKK$).
The adjunction 
\[ 
\xymatrix{
\Mod_{S^0} \ar@<.5ex>[r] & 
\Mod_\caKK  \ar@<.5ex>[l]}
\]
induces an adjunction with a symmetric monoidal left adjoint functor
\[ 
\xymatrix{\mathsf{Sym}_{S^0}(\Sigma^\infty\P^1)-\Mod_{S^0}^\Sigma \ar@<.5ex>[r]  & 
\mathsf{Sym}_\caKK(\Sigma^\infty\P^1\wedge \caKK)-\Mod^\Sigma_\caKK:f. \ar@<.5ex>[l]}
\]
Observing the left hand side of this adjunction is the category of motivic symmetric spectra, 
we obtain a (strictly) commutative motivic ring spectrum $f(\mathsf{Sym}_\caKK(\caKK))$. 
This is the desired delooping of $\caKK$.
Note that $f(\mathsf{Sym}_\caKK(\caKK))$ and $\KGL$ are equivalent as motivic symmetric spectra
because they are both $\Omega-\mathbb{P}^1$-spectra with terms $\caKK\simeq K$, 
and the bonding maps of the former are by definition identified with those of the latter.
We claim this can be refined to an equivalence between homotopy commutative motivic ring spectra.

Since $\KGL$ and $f(\mathsf{Sym}_\caKK(\caKK))$ have the same motivic homotopy type, 
$f(\mathsf{Sym}_\caKK(\caKK))$ gives rise to a second multiplication on $\KGL$, 
which we want to check agrees with the standard one up to homotopy. 
That is, 
we show their difference $\delta\in\KGL^{00}(\KGL\wedge\KGL)$ vanishes (an analogous and easier argument takes care of the unit maps).
There are no nontrivial phantom maps $\KGL\wedge\KGL\to\KGL$ \cite[Remark 9.8 (ii)]{motiviclandweber}, 
so it suffices that $\delta$ maps to $0$ when pulled back along any map $T\to\KGL\wedge\KGL$, 
where $T$ is compact. 

We have
\[ 
\KGL 
= 
\colim_n\, \Sigma^{-2n,-n}\Sigma^\infty_{\mathbb{P}^1}(\mathbb{Z}\times\BGL), 
\]
which immediately implies 
\[ 
\KGL\wedge\KGL  
= 
\colim_{n,m}\, \Sigma^{-2n,-n}\Sigma^\infty_{\mathbb{P}^1}(\mathbb{Z}\times\BGL) \wedge \Sigma^{-2m,-m}\Sigma^\infty_{\mathbb{P}^1}(\mathbb{Z}\times\BGL). 
\]
Thus every map $T\to \KGL\wedge\KGL$ factors through one of the canonical maps
\[ 
\iota_{n,m}: 
\Sigma^{-2n,-n}\Sigma^\infty_{\mathbb{P}^1}(\mathbb{Z}\times\BGL) \wedge \Sigma^{-2m,-m}\Sigma^\infty_{\mathbb{P}^1}(\mathbb{Z}\times\BGL) 
\to 
\KGL\wedge\KGL,
\]
and it suffices to show that $\delta\circ\iota_{n,m}=0$ for all $n,m\ge 0$.
The next lemma, 
which follows easily from the constructions of the multiplications, 
reduces to the case $n=m=0$.

\begin{lemma} 
\label{lemma:multiplicationbott} 
Writing the $\mathbb{P}^1$-spectrum $\KGL=(K,K,\ldots)$ as a sequence of $S^1$-spectra, 
the identity on $K$ determines a map $\alpha:\KGL\to\Sigma^{2,1}\KGL$.
The latter map agrees with multiplication by the Bott element for both of the given ring structures on $\KGL$.
\end{lemma}

From the above one deduces the commutativity of the diagram:
\[ 
\xymatrix{ 
& & \KGL\ar[dd]^-{\alpha} \\ 
\Sigma_{\mathbb{P}^1}^\infty(\mathbb{Z}\times\BGL) \ar[rru]^-{\iota_0} \ar[rrd]^-{\iota_1} & & \\
& & \Sigma^{2,1}\KGL  
} 
\]
By iteration, 
the left hand square in the diagram 
\[ 
\xymatrix{ 
\Sigma^\infty_{\mathbb{P}^1}(\mathbb{Z}\times\BGL)\wedge \Sigma^\infty_{\mathbb{P}^1}(\mathbb{Z}\times\BGL) \ar[r]^-{\iota_{0,0}} \ar[d]^-{id} 
& \KGL\wedge\KGL \ar[d]^{\alpha^n\wedge\alpha^m}_{\simeq} \ar[r]^-{\delta} & \KGL \ar[d]^-{\alpha^{n+m}}_{\simeq}\\
\Sigma^\infty_{\mathbb{P}^1}(\mathbb{Z}\times\BGL)\wedge \Sigma^\infty_{\mathbb{P}^1}(\mathbb{Z}\times\BGL) \ar[r]^-{\iota_{n,m}} 
& \Sigma^{2n,n}\KGL \wedge \Sigma^{2m,m}\KGL  \ar[r]^-{\delta} & \Sigma^{2(n+m),n+m}\KGL 
} 
\]
commutes (with the suspension coordinate for $\iota_{n,m}$ omitted).
Lemma \ref{lemma:multiplicationbott} implies the right hand square commutes.
This concludes the reduction to the case $n=m=0$, 
i.e., 
it suffices to consider the canonical map out of 
\[ 
\Sigma_{\P^1}^\infty\left(\Sigma_{S^1}^\infty\left( (\Z\times\BGL)\wedge (\Z\times\BGL)\right)\right)\]
\noindent
by viewing motivic symmetric spectra as motivic symmetric spectra in $S^1$-spectra. 
By adjunction the assertion is clear now because by assumption there are equivalences of commutative monoids in the homotopy category of $S^1$-spectra
\[
K
=
\Omega^\infty_{\P^1}(\KGL)\simeq \Omega^\infty_{\P^1}(f(\mathsf{Sym}_\caKK(\caKK)))
=
\caKK. 
\]
\noindent
The uniqueness clause of Theorem \ref{uniqueKGL} implies $f(\mathsf{Sym}_\caKK(\caKK))$ and 
$\KGL$ are equivalent as commutative motivic ring spectra.
We obtain the desired equivalence of $E_\infty$ $S^1$-spectra by passing to loop spaces
(the first equivalence is due to Morel-Voevodsky and uses regularity of $S$)
\[ 
K
\simeq
\Omega^\infty_{\P^1}(\KGL)
\simeq
\Omega_{\P^1}^\infty(f(\mathsf{Sym}_\caKK(\caKK)))
\simeq\caKK.
\]
\end{proof}

\begin{remark} 
The idea of delooping an $E_\infty$ structure, 
due to the second author, 
can be applied in other contexts. 
For example, 
the uniqueness of the $E_\infty$ structure on $\KU$ implies there is a unique $E_\infty$ space structure on
$\Z\times\BU$ in pointed spaces, 
which refines the usual multiplicative structure up to homotopy.
The analogous observation applies to $\KGL$ and $\Z\times\BGL$.
A more elaborate version of the argument can be used to construct an $E_\infty$ structure on the motivic Hermitian $K$-theory spectrum.
We expect uniqueness of such an $E_\infty$ structure when the base scheme contains a square root of $-1$,
because then the cooperations are base changed from topological real $K$-theory.
The details will be the subject of a forthcoming PhD thesis.
\end{remark}

{\bf Acknowledgements.}
The main result of this paper was announced by the first named author at the 2009 M{\"u}nster workshop 
on Motivic Homotopy Theory. 
He thanks the organizers E.~M.~Friedlander, G.~Quick and P.~A.~{\O}stv{\ae}r for the invitation, 
and P.~A.~{\O}stv{\ae}r for hospitality while visiting the University of Oslo, where the major part of this work was finalized.
The authors thank A.~Robinson for valuable discussions keeping us informed about his recent work, 
an anonymous referee and J.~Rognes for comments on a shorter version of this paper and another referee for
a careful inspection of the final section.

\bibliographystyle{plain}
\bibliography{motobs}

\vspace{0.1in}

\begin{center}
Fakult{\"a}t f{\"u}r Mathematik, Universit{\"a}t Regensburg, Germany.\\
e-mail: niko.naumann@mathematik.uni-regensburg.de
\end{center}
\begin{center}
Institut f{\"u}r Mathematik, Universit{\"a}t Osnabr{\"u}ck, Germany.\\
e-mail: markus.spitzweck@uni-osnabrueck.de
\end{center}
\begin{center}
Department of Mathematics, University of Oslo, Norway.\\
e-mail: paularne@math.uio.no
\end{center}
\end{document}